\documentclass[12pt]{amsart}

\usepackage{extsizes}
\usepackage{amsmath, amsthm, amssymb}
\usepackage{graphics}
\usepackage{pinlabel}
\usepackage{color}
\usepackage[colorlinks,citecolor=blue,linkcolor=red]{hyperref}
\usepackage{graphicx}
\usepackage{nicefrac}						
\usepackage{mathtools}
\usepackage{float}
\usepackage{tikz-cd}
\usepackage{enumerate}
\usepackage{xcolor}

\usepackage{geometry}
\usepackage{secdot}
\usepackage{caption}

\usepackage{marginnote}
 
\definecolor{ao(english)}{rgb}{0.0, 0.5, 0.0}
\definecolor{pale}{HTML}{FF9955}

\usepackage{overpic}

\DeclareMathOperator{\Mod}{Mod}

\newtheorem{theorem}{Theorem}[section]
\newtheorem{lemma}[theorem]{Lemma}
\newtheorem{remark}[theorem]{Remark}
\newtheorem{definition}[theorem]{Definition}
\newtheorem{corollary}[theorem]{Corollary}

\newtheorem{proposition}[theorem]{Proposition}
\newtheorem{question}{Question}

\begin{document}

\title{Large volume fibered knots in 3-manifolds}

\author{J. Robert Oakley}

%    Abstract is required.
 \begin{abstract}
We prove that for hyperbolic fibered knots in any closed, connected, oriented 3-manifold the volume and genus are unrelated. As an application we answer a question of Hirose, Kalfagianni, and Kin about volumes of mapping tori that are double branched covers.
\end{abstract}

\maketitle

\section{Introduction}
Let $M$ denote a closed, connected, oriented 3-manifold. Alexander in \cite{Alexander} proved that every such $M$ contains a fibered link. This result has been strengthened in many ways such as the link being a knot \cite{Myers} and the monodromy being right-veering and pseudo-Anosov \cite{CH}. More recently, this construction has been used in \cite{FPS} to show that for hyperbolic fibered knots in the 3-sphere, volume and genus are unrelated. In this paper we generalize the approach in \cite{FPS} to show that for hyperbolic, fibered knots in closed, connected, oriented 3-manifolds the volume and genus are unrelated. 
In particular we prove the following theorem.
\begin{theorem} \label{mainthm}
Let $M$ be a closed, connected, oriented 3-manifold. There exists some $g_0>1$ such that for all $g\geq g_0$ and $V>0$ there exists a knot $K\subseteq M$ such that $M - K$ is fibered over the circle with genus $g$ and $M - K$ is hyperbolic with $\text{vol}(M - K) > V.$
\end{theorem}
  
As an application of theorem \ref{mainthm} we answer a question asked by Hirose, Kalfagianni, and Kin in \cite{HKK}. Let $\mathfrak{D}_g(M) \subseteq \Mod(S_g)$ be the subset of the mapping class group of a closed surface of genus $g$ consisting of elements whose mapping tori are 2-fold branched covers of $M$ branched along a link. Hirose, Kalfagianni, and Kin asked the following question in \cite{HKK}.

\begin{question}\label{qn}
For $g$ sufficiently large, does $\mathfrak{D}_g(M)$ contain an infinite family of pseudo-Anosov mapping classes whose mapping tori have arbitrarily large volume?
\end{question}

Combining theorem \ref{mainthm} and theorem 10 of \cite{HKK} yields the following affirmative answer to question \ref{qn}.

\begin{corollary}
For any closed, connected, oriented 3-manifold $M$ there exists some $g_0>1$ such that for any $g>g_0$ the set $\mathfrak{D}_{2g}(M)$ contains an infinite family of pseudo-Anosov elements whose mapping tori have arbitrarily large volume.
\end{corollary}

\subsection{Outline of the proof}
The strategy of the proof is to construct for every closed, connected, oriented 3-manifold $M$ and for every sufficiently large genus $g$ of the fiber, a family of fibered hyperbolic knots of genus $g$ with linearly growing volume. We use three main tools to construct these families. The first is the theory of branched covers in dimensions 2 and 3. Using branched covers of the 3-sphere branched over a knot or link in braid position to obtain a fibered link in a 3-manifold goes back to Alexander. Here we use an improvement on Alexander's theorem due to Hilden and Montesinos \cite{Hilden, Montesinos}. This refinement allows us to consider simple 3-fold branched covers. That the branched covers are degree 3 is crucial for control of the number of preimages of the braid axis which will become the fibered link. This control over the degree of the cover comes at the expense of the covers considered by Hilden and Montesinos being irregular. Fortunately these irregular covers have the property of being \emph{simple} (see definition \ref{simple}). Simple branched covers are well studied in dimension 2 (see for example \cite{Fuller, GK, Winarski}). This allows us to control the monodromy of the fibered link.

In section \ref{sec3} we use this control of the monodromy and the second tool, open book decompositions and stabilization, to ensure that our fibered link becomes a fibered knot with pseudo-Anosov monodromy (see section \ref{sec3} for a brief discussion of open book decompositions). To that end we use the equivalence between open book decompositions and fibered links as well as the technique introduced by Colin and Honda in \cite{CH} to transform our fibered 2-component link with possibly reducible monodromy, obtained from the branched cover construction described above, into a fibered knot with pseudo-Anosov monodromy.

Crucially, even after this transformation we maintain the control over the monodromy. This allows us to use our third tool, subsurface projections and Brock's work in \cite{Brock} relating translation distance of a pseudo-Anosov in the pants graph and the volume of the associated mapping torus. The control over the monodromy that we have maintained throughout the construction of these fibered hyperbolic knots ensures that in addition to the monodromies being pseudo-Anosov, they factor as $(A)(S^{-1}F^nS)(F^{-n})$ where $A$ is constant as $n$ varies and $F^{\pm n}$ are pseudo-Anosov on essential subsurfaces. Then we apply the work of Clay--Leininger--Mangahas \cite{CLM} to see that the monodromy has linearly growing subsurface projections to the supports of $F^{-n}$ and $S^{-1}F^nS$. The Masur--Minsky distance formula for the pants graph \cite{MMII} then implies that the monodromies have linearly growing translation distance in the pants graph. Finally, the work of Brock \cite{Brock} implies that this linearly growing translation distance corresponds to linearly growing volume of the associated mapping tori.

\subsection{Acknowledgments}
The author thanks his advisor Dave Futer for pointing him towards this problem as well as his guidance. The author thanks the NSF for its support via grant DMS-1907708. The author also thanks the anonymous referee for comments which expanded and improved the exposition of this paper.

\section{Background}
Throughout the paper $S = S_{g,n}$ denotes an orientable surface of genus $g$ and $n$ boundary components. We will now introduce some conventions and aspects of the study of surfaces and their symmetries that will be useful. The starting point of our consruction will use the theory of braids in the 3-sphere. We leverage the group theoretic structure of braids.

\begin{definition}
    The \textbf{Braid Group on n-strands,} $\mathfrak{B}_n$ is the group of isotopy classes of $n-\text{braids}.$ If $\eta, \psi \in \mathfrak{B}_n,$ then we multiply $\eta$ and $\psi$ by scaling each of their heights by $1/2$ and stack the braid corresponding to $\eta$ on top of the braid corresponding to $\psi.$ The resulting braid is $\eta\cdot \psi \in \mathfrak{B}_n.$
\end{definition}

The braid group, $\mathfrak{B}_n$ is generated by the braids $\sigma_i$ for $1\leq i \leq n-1$ where the only crossing is the $(i+1)^{th}$ strand passing over the $i^{th}$ strand (see figure \ref{fig:braid_gen}).

\begin{figure}
    \centering
    \begin{overpic}[width=4cm,]{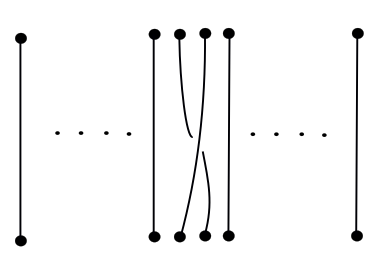}
    \put(4,65){\tiny 1}
    \put(94,65){\tiny $n$}
    \put(46,65){\tiny i}
    \put(50,65){\tiny i+1}
        
    \end{overpic}
    \caption{The generator $\sigma_i$ of the braid group on $n$ strands.}
    \label{fig:braid_gen}
\end{figure}

\subsection{Mapping Class Group Basics}

\begin{definition}
    The \textbf{Mapping Class Group} of a surface, $\Mod(S_{g,n}),$ is $\text{Homeo}^+(S,\partial S) / \text{isotopy}.$ This is the group of orientation preserving homeomorphisms of the surface $S_{g,n}$ that restrict to the identity on the boundary, up to isotopy. 
\end{definition}

For composing elements of $\Mod(S_{g,n}),$ called mapping classes, we us funtional notation. That is for $\varphi, \psi \in \Mod(S_{g,n})$ the composition $\varphi \circ \psi$ refers to applying $\psi$ and then applying $\varphi.$

\begin{remark}
    We note that we will occasionally take the perspective that $\mathfrak{B}_n$ is $\Mod(D_n)$ the mapping class group of the disk with $n$ punctures. See chapter 9 of \cite{primer}.
\end{remark}

A simple example of a non-trivial element of the mapping class group is a \textbf{Dehn twist.}
Let $\alpha \subseteq S$ be an essential simple closed curve on a surface $S.$ The Dehn twist about $\alpha,$ denoted $T_{\alpha},$ is the mapping class obtained by cutting $S$ along $\alpha,$ twisting a neighborhood of one boundary component by an angle of $2\pi$, and then regluing the surface. See figure \ref{fig:dehn_twist} for an illustration. Dehn twists are an important aspect of the theory of mapping class groups that we will use extensively. In fact, $\Mod(S)$ is generated by a finite collection of Dehn twists about simple closed curves. See chapter 3 of \cite{primer} for an extensive discussion of Dehn twists.

\begin{figure}[H]
    \centering
    \begin{overpic}[width=5cm,]{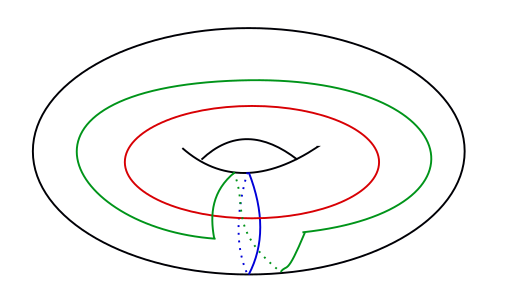}
    \put(45,47){\color{ao(english)}\tiny $T_{\alpha}(\beta)$}
    \put(65,25){\color{red} \tiny $\beta$}
    \put(45,2){\color{blue} \tiny $\alpha$}
        
    \end{overpic}
    \caption{The action of the Dehn twist $T_{\alpha}$ on the curve $\beta.$}
    \label{fig:dehn_twist}
\end{figure}

In studying the mapping class group we consider different classes of mapping classes depending on their action on the surface. We say that a mapping class, $f\in \Mod(S_{g,n})$ is \textbf{reducible} if there is a nonempty set of isotopy classes of mutually disjoint, simple closed curves, $\{a_1, \dots, a_n\}$ such that $\{a_1, \dots, a_n\} = \{f(a_1), \dots, f(a_n)\}.$ We say that a mapping class is \textbf{periodic} if it is finite order. We say that a mapping class $f\in \Mod(S_{g,n})$ is \textbf{pseudo-Anosov} if there exists a pair of transverse measured foliations $(\mathcal{F}^u, \mu_u)$ and $(\mathcal{F}^s, \mu_s)$ on $S_{g,n},$ a number $\lambda > 1,$ and a representative homeomorphism $\varphi$ such that
\[\varphi (\mathcal{F}^u, \mu_u) = (\mathcal{F}^u, \lambda \mu_u) \quad \text{and} \quad \varphi (\mathcal{F}^s, \mu_s) = (\mathcal{F}^s, \lambda^{-1} \mu_s).\] 
See chapter 13 of \cite{primer} for more discussion of the three types of mapping classes. The following theorem says that these three classes characterize all elements of the mapping class group.

\begin{theorem}\textbf{(Nielsen-Thurston Classification)}
    Every mapping class in $\Mod(S_{g,n})$ is either periodic, reducible, or pseudo-Anosov. Moreover, pseudo-Anosov mapping classes are neither periodic nor reducible.
\end{theorem}

An interesting sub-class of reducible mapping classes are \textbf{partial pseudo-Anosovs.}
A mapping class $f \in \Mod(S)$ is called a partial pseudo-Anosov if it satisfies the following conditions.
\begin{enumerate}
    \item There exists a subsurface $S'\subseteq S$ such that $f|_{S'}$ is a pseudo-Anosov.
    \item $f|_{S - S'}$ is isotopic to the identity.
\end{enumerate}
In this case, we call $S'$ the support of $f.$

\subsection{The Curve Graph}\label{sec: curves}
Given two isotopy classes of arcs or curves $a,b \subseteq S,$ the \emph{geometric intersection number} of $a$ and $b$ is defined to be the minimal number of intersections between any pair of curves or arcs representing each isotopy class. That is,
\[i(a,b) = \min_{\alpha\in a, \beta\in b} \lvert \alpha \cap \beta \rvert\]

\begin{definition}
    Let $3g+n \geq 5.$ Then the \textbf{Curve Graph of a surface $S = S_{g,n},$} denoted $\mathcal{C}(S),$ is the simplicial graph with vertices corresponding to isotopy classes of essential simple closed curves and whose edges represent pairs of curves that can be realized disjointly on $S.$ Note that there is an adaption of the definition for surfaces of lower complexity by modifying the edge condition, but all of the surfaces we discuss here satisfy the complexity condition (see \cite[Chapter 4]{primer}).
\end{definition}

Note that we will use $\mathcal{C}^{(0)}(S)$ to denote the $0-$skeleton of $\mathcal{C}(S).$ We endow $\mathcal{C}(S)$ with a metric by assigning each edge length $1.$ With this metric, $d_S,$ the curve graph is a geodesic metric space. There is a natural action of $\Mod(S),$ by isometries, on $\mathcal{C}(S).$ Given $f\in\Mod(S)$ we define the \textit{translation distance of f, $\tau_{\mathcal{C}}(f)$} as:
\[\tau_{\mathcal{C}}(f) = \inf_{\eta\in\mathcal{C}^{(0)}(S)}\{d_S(\eta, f(\eta)\}\]

We now state the following lemma originally due to Hempel in \cite{Hempel}.

\begin{lemma} \label{Hempel_lemma}
    Let $\alpha,\beta \subseteq S$ be simple closed curves representing vertices of $\mathcal{C}(S),$ $a$ and $b$ respectively. Then
    \[d_S(a,b) \leq 2\log_2(i(\alpha,\beta)) + 2\]
\end{lemma}

The theory of subsurface projections plays a central role in the theory of mapping class groups, and we introduce the basics of this here. For more discussion of subsurface projections, see section 2 of \cite{MMII}. An \textit{essential subsurface} $Y\subseteq S$ is a subsurface such that the interior of $Y$ embeds in $S$ so that each component of $\partial Y$ maps to either an essential simple closed curve in $S$ or a component of $\partial S.$ A \textit{non-annular} subsurface is a subsurface that is not homeomorphic to an annulus. An essential simple closed arc or curve, $\eta,$ is in \emph{minimal position} with $Y\subseteq S$ when $\eta$ is in minimal position with all of the components of $\partial Y.$ That is, $i(\eta,C)$ is minimal for all components $C$ of $\partial Y.$

Given $Y\subseteq S$ a non-annular essential subsurface, the \emph{subsurface projection}
\[\pi_Y: \mathcal{C}(S) \longrightarrow \mathcal{C}(Y)\]
takes a vertex $a\in \mathcal{C}^{(0)}(S)$ to a subset of $\mathcal{C}(Y)$ as follows. First, put $a$ and $Y$ in minimal position. If $a\subseteq Y$ then $\pi_Y(a) = a\in \mathcal{C}(Y).$ If $a$ intersects $Y$ in a collection of arcs we associate to the collection of arcs a collection of essential simple closed curves as described below. Suppose $\alpha$ is an arc in $Y.$ We associate to $\alpha$ a collection of pairwise disjoint simple closed curves in $Y,$ denoted $\overline{\alpha}.$ Concretely, let $N(\alpha)$ denote a thickening of the union of $\alpha$ with the (at most two) components of $\partial Y$ that it meets. We then define $\overline{\alpha}$ to be the essential components of $\partial N(\alpha).$ See figure \ref{fig:aos} for an example where the arc meets two boundary components.
Thus, if $a$ intersects $Y$ in a collection of arcs, we define $\pi_Y(a)$ to be the collection of essential simple closed curves $\overline{a}$ obtained as described above.
Now, given $a,b\in\mathcal{C}^{(0)}(S)$ we define $d_Y(a,b)$ by
\[d_Y(a,b) = \text{diam}_{\mathcal{C}(Y)}(\pi_Y(a), \pi_Y(b))\]
Finally, we note that $\pi_Y$ is coarsely Lipschitz (lemma 2.3 of \cite{MMII}):
\[d_Y(a,b) \leq 2\cdot d_S(a,b)+2.\]

The following is known as the bounded geodesic image theorem. It is due to Masur and Minsky in \cite{MMII} and appears as Theorem 3.1 there.

\begin{theorem}\textbf{(Bounded Geodesic Image Theorem)} \label{bgit}
    There exists a constant $M$ depending only on $\chi(S)$ such that the following holds. Let $Y\subseteq S$ be an essential subsurface. If $\xi$ is a geodesic in $\mathcal{C}(S)$ all of whose vertices represent curves which intersect $Y$ nontrivially, then $\text{diam}(\pi_Y(\xi)) \leq M.$ 
\end{theorem}
\begin{remark}\label{bgit_rmk}
    We will use the contrapositive of the above: If the projected image of $\xi$ to $\mathcal{C}(Y)$ is larger than $M$, then there is a vertex on the geodesic, $\xi,$ that represents a curve disjoint from $Y.$ In particular, said vertex has distance 1 to $\partial Y.$ 
\end{remark}

\subsection{The Pants Graph}
We will use another graph associated to the curves on a surface called the pants graph. Before we define the graph we need the following definitions.

A \textit{pair of pants} is a compact surface of genus 0 with three boundary components. Let $S$ be a surface with $\chi(S)<0.$ A \textit{pants decomposition} of $S$ is a collection of disjoint simple closed curves on $S$ such that the result of cutting $S$ along these curves results in a disjoint union of pairs of pants. Two pants decompositions, $P, P'\subseteq S$ are related by an \textit{elementary move} if $P'$ can be obtained from $P$ by replacing a curve $\alpha\in P$ with a curve $\beta \neq \alpha$ such that $i(\alpha, \beta)$ is minimal over all choices of $\beta$ (i.e. over all $\beta$ such that replacing $\alpha$ with $\beta$ results in a pants decomposition). We are now equipped to define the pants graph.

\begin{definition}
    The \textbf{Pants Graph of $S$,} denoted $\mathcal{P}(S),$ is the graph with vertices corresponding to isotopy classes of pants decompositions of $S$ and whose edges connect pants decompositions which differ by an elementary move.
\end{definition}
The graph, $\mathcal{P}(S)$ also carries a natural metric, $d_{\textbf{P}},$ by assigning each edge length 1. Again, the mapping class groups naturally acts by isometries on the pants graph. Using this metric structure we can define the translation distance of a mapping class on the pants graph, $\tau_{\textbf{P}}(f),$ in the same way as in the curve graph.

The following is an important theorem which relates the distance between pants decompositions in the curve graph of a subsurface to distance in the pants graph. It is commonly referred to as the Masur--Minsky distance formula for the pants graph, and appears originally as theorem 6.12 of \cite{MMII}. It appears, as written, as theorem 4.4 of \cite{Brock2}.

\begin{theorem}\label{MM_dist}
    There is a constant, $M_0 = M_0(S) >0$ such that, given $M \geq M_0,$ there are constants $e_0, e_1,$ for which, if $P,P'\subseteq S$ are any two pants decompositions of $S,$ then
    \[e_0^{-1}d_{\textbf{P}}(P,P') - e_1 \leq \sum d_Y(P,P') \leq e_0d_{\textbf{P}}(P,P') + e_1,\]
    where the sum is taken over all non-annular subsurfaces, $Y\subseteq S,$ such that $d_Y(P,P') \geq M.$
\end{theorem}

The primary use we have for the pants graph is the close relation between the geometry of a mapping torus and the action of the monodromy on the pants graph. In particular, we make use of the following theorem of Brock (see \cite{Brock}).
\begin{theorem}\label{brock}
    Given $S$ there is a constant $K>1$ so that if $M_f$ is the mapping torus of a pseudo-Anosov, $f\in \Mod(S),$ then
    \[K^{-1}\tau_{\textbf{P}}(f) \leq vol(M_f) \leq K\tau_{\textbf{P}}(f).\]
\end{theorem}

\subsection{Geodesic Laminations}
We now introduce some of the theory of geodesic laminations that we will use later. For a thorough exposition of the theory, the reader should see \cite{Bonahon, CB, FLP}.

\begin{definition}
    Fix a surface, $S,$ with $\chi(S)<0$ and fix a hyperbolic metric on $S.$ A \textbf{geodesic measured lamination}, $L,$ on $S$ is a non-empty collection of disjoint simple geodesics whose union is closed, along with a transverse measure that is invariant as you flow along the geodesics.
\end{definition}
We denote the space of geodesic measured laminations on $S$ by $\mathcal{ML}(S).$ We now also define the space of projective measured laminations, $\mathcal{PML}(S),$ as the space $\mathcal{ML}(S)$ modulo the scaling of measures. We now require the following definitions.
\begin{definition}
A geodesic lamination $L$ is \textbf{minimal} if $L$ does not contain any proper sublaminations. A lamination, $L,$ is \textbf{filling} if every component of $S - L$ is either an ideal polygon or a once punctured ideal polygon.
\end{definition}

We now give a second, equivalent, definition of what it means for a homeomorphism to be pseudo-Anosov.

\begin{definition}
  A homeomorphism, $f,$ of $S$ is pseudo-Anosov if it preserves a pair of transverse measured laminations on $S,$ one expanding and one contracting. We call these laminations the \textbf{unstable} and \textbf{stable laminations} of $f$ respectively.
\end{definition}

Recall that, at the beginning of this section, a pseudo-Anosov homeomorphism was defined in terms of stable and unstable foliations rather than the stable and unstable laminations as above. In fact, there is an exact correspondence between measured laminations and measured foliations. Further, the stable lamination of a pseudo-Anosov, $f,$ corresponds to the stable foliation of $f$ under this correspondence. For a precise treatment of the correspondence, see \cite{Levitt}.

Masur and Minsky showed in \cite{MMI} that the curve graph is Gromov hyperbolic. Hence, we may consider the boundary at infinity, $\partial_{\infty}\mathcal{C}(S).$ Klarreich showed in \cite{Klar} that $\partial_{\infty}\mathcal{C}(S)$ is the space of minimal filling laminations, $\mathcal{EL}(S)\subseteq \mathcal{ML}(S).$ Moreover, a sequence of curves $\alpha_i \in \mathcal{C}(S)$ converges to a point $\alpha\in \partial_{\infty}\mathcal{C}(S)$ if and only if the sequence converges to $\alpha \in \mathcal{EL}(S)$ equipped with the subspace topology inherited from $\mathcal{PML}(S).$

We also note here that the geometric intersection number as defined above for arcs and curves extends uniquely to a continuous, homogeneous function $i: \mathcal{ML}(S) \times \mathcal{ML}(S) \longrightarrow \mathbb{R}$ that we also call the geometric intersection number (see \cite[chapter 9]{ThurstonNotes}).

\section{Branched Cover Construction}
Let $M$ be a closed, connected, oriented 3-manifold. In this section we begin our construction of the desired knots in $M.$ We start by constructing a fibered 2-component link in $M$ using branched covers. We need the following definition.

\begin{definition}\label{simple}
A branched covering $p:M \longrightarrow N$ of degree $d$ is called \textbf{simple} if each point in $N$ has at least $d-1$ preimages in $M$. In particular, points away from the branch locus in $N$ will have $d$ preimages while points on the branch locus in $N$ will have $d-1$ preimages in $M.$ 
\end{definition}

Due to a theorem of Hilden and Montesinos, \cite{Hilden, Montesinos}, there exists a 3-fold simple branched covering $p:M\longrightarrow S^3$ branched along a knot $K.$ Since the branched covers constructed are simple, points on the branch locus $K\subseteq S^3$ have two preimages in $M.$ One preimage has branching index 1 while the other has branching index 2. By a theorem of Alexander \cite{Alexander} we can represent $K$ as the closure of a braid $\Pi\in\mathfrak{B}_k$ for some $k.$

\begin{figure}
    \centering
    \begin{overpic}[width = 6cm]{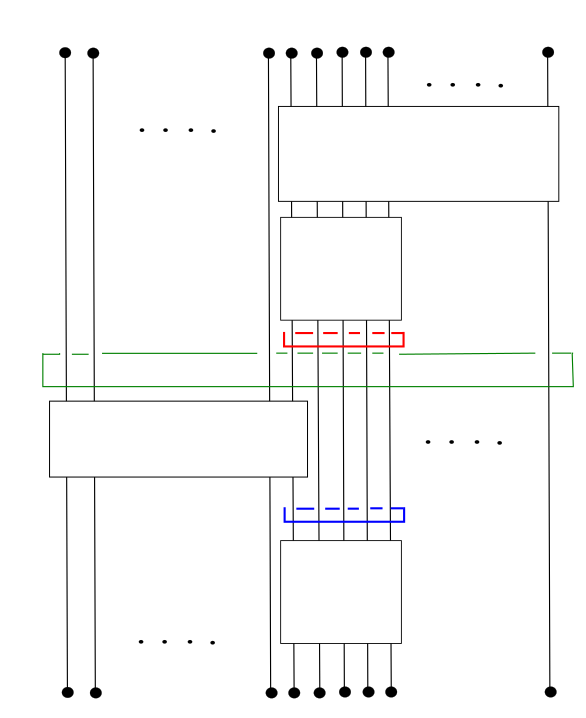}
        \put(8.25,95){1}
        \put(25,95){$2g+1-k$}
        \put(70, 95){$2g+1$}
        \put(55,74){\huge$\Pi$}
        \put(43,60){\LARGE $\Phi^n$}
        \put(57,53){\color{red}$\gamma_n$}
        \put(-1,48){\color{ao(english)} $\omega_n$}
        \put(21,35){\huge$\Sigma$}
        \put(57,27){\color{blue}$\delta_n$}
        \put(40,14){\LARGE $\Phi^{-n}$}
    \end{overpic}
    \caption{The braid whose closure will be the branching locus of our branched cover. The braid $\Pi$ is the braid word representing the knot $K$ which we branch over. The braid $\Phi$ is a pseudo-Anosov on its support. The word $\Sigma$ totally consists of stabilizations. Note that $\gamma_n$ and $\delta_n$ and the disks they bound do not change as $n$ increases.}
    \label{fig:Braid}
\end{figure}

Fix $g>1$ such that $2g+1 > k.$ Let $\sigma_i$ denote the positive half-twist between the $i^{\text{th}}$ and $(i+1)^{\text{th}}$ strands. We will consider $\Pi\in\mathfrak{B}_{k}$ under the natural inclusion into $\mathfrak{B}_{2g+1}$ where $\Pi$ is on the last $k$ strands (see figure \ref{fig:Braid}). In a variation on the construction of \cite{FPS} we now define the following braids.
\begin{align*}
    \Phi &= (\sigma_{2g+2-k})^{-1}(\sigma_{2g+3-k})(\sigma_{2g+4-k})^{-1}(\sigma_{2g+5-k}) \\
    \Sigma &= (\sigma_{2g+1-k})(\sigma_{2g-k})^{-1} \cdots (\sigma_{2})^{\pm}(\sigma_{1})^{\mp}\\
    \beta_n &= \Pi\Phi^n\Sigma\Phi^{-n}
\end{align*}

Let $\widehat{\beta_n}$ be the braid closure of $\beta_n.$ Let $\omega_n$ be its braid axis encircling the braid after the $\Phi^n$ factor and before the $\Sigma$ factor of $\beta_n.$ Let $\Lambda_n = \widehat{\beta_n} \cup \omega_n.$ Note that $\omega_n$ bounds a disk $\Omega_n$ in $S^3$ meeting $\widehat{\beta_n}$ in $2g+1$ points. Hence, $S^3 - \Lambda_n$ is a punctured disk bundle over the circle with monodromy $\beta_n.$

\begin{lemma}
The braid closure $\widehat{\beta_n}$ is the knot $K.$
\end{lemma}

\begin{proof}
Note that $\widehat{\Pi}$ is $K,$ and that $\widehat{\beta_n}$ is stabilized along the first $2g+1-k$ strands. Destabilizing smooths the crossings coming from $\Sigma$ and deletes the first $2g+1-k$ strands. Now the $\Phi^n$ and $\Phi^{-n}$ factors cancel. We are left with the $k$-strand braid closure $\widehat{\Pi}$ which by definition is $K.$
\end{proof}
We will now introduce some notation. Let $\gamma_n$ and $\delta_n$ be the simple closed curves encircling the $2g+2-k,\: 2g+3-k,\: 2g+4-k,\: 2g+5-k,\: \text{and}\: 2g+6-k$ strands of $\beta_n$ taken immediately before and after the $\Sigma$ factor of $\beta_n$ (see figure \ref{fig:Braid}). Let $\Gamma_n$ and $\Delta_n$ denote the disks bounded by $\gamma_n$ and $\delta_n$ respectively. Note that $\Gamma_n$ and $\Delta_n$ each intersect $\beta_n$ in 5 points. Let $w_n \: \text{and}\: W_n$ denote $p^{-1}(\omega_n)\: \text{and}\: p^{-1}(\Omega_n)$ respectively.

\begin{lemma}
   The manifold $N_n = M - w_n$ is an $S_{g-1,2}-$bundle over the circle.
\end{lemma}

\begin{proof}
That $N_n$ is fibered over the circle follows from the fact that $S^3 - \Lambda_n$ is a punctured disk bundle over the circle. Since $p$ is a simple, 3-fold cover and $\Omega_n$ intersects $\beta_n$ in $2g+1$ points, an Euler characteristic calculation shows that $W_n$ is homeomorphic to $S_{g-1,2}.$ Indeed, $p$ restricts to a simple 3-fold branched cover of the disk, $\Omega_n,$ branched over the $2g+1$ points of intersection with $\beta_n.$ By the definition of a 3-fold simple branched cover, each of the $2g+1$ branch points in $\Omega_n$ have two preimages. One of the preimages has branching index one while the other has branching index two. This is sufficient to compute the Euler characteristic of the fiber surface of $N_n.$ First, we apply the Riemann-Hurwitz formula:
\[\chi(W_n) = 3\cdot \chi(\Omega_n) - \sum_{p\in W_n} (e_p -1)\] where $e_p$ denotes the branching index at $p.$ As discussed above, we know the branching indices of all points in $W_n,$ so 
\[\chi(W_n) = 3\cdot 1 - (2g+1).\]
Hence, $\chi(W_n) = 2-2g.$ The boundary of $\Omega_n$ is triply covered by the boundary of $W_n.$ Therefore, $\partial W_n$ has 1, 2, or 3 components. By the above calculation $\chi(W_n)$ is even. Thus, $\partial W_n$ has 2 components. We also note here that by the definition a simple branched cover, $W_n$ is connected. Indeed, if $W_n$ were not connected, then it must be the disjoint union of two or three homeomorphic surfaces. By the above, $W_n$ has two boundary components, so $W_n$ must be the disjoint union of two surfaces, $W_1$ and $W_2$, each with one boundary component. However, the branched covering map, $p:W_n \rightarrow \Omega_n,$ restricts to a branched cover of $\Omega_n$ on each of $W_1$ and $W_2.$ In particular, we may assume without loss of generality that $p|_{W_1}$ is a one-sheeted cover and $p|_{W_2}$ is a two-sheeted cover. In this case, $p|_{W_1}$ is a homeomorphism, and thus $\Omega_n \cong W_1 \cong W_2.$ This contradicts the above Euler characteristic calculation for $W_n$ since $g>1.$ Combining the above, we see that $W_n$ is homeomorphic to, $S_{g-1,2},$ a surface with genus $g-1$ and 2 boundary components.

\end{proof}

\section{Hyperbolicity and Volume} \label{sec3}
To prove the main theorem we need to construct from $N_n$ the desired knot complements, verify their hyperbolicity, and show that their volumes grow linearly with $n$. In service of this we will change our perspective to a 2-dimensional perspective. We will focus on the fiber surface $W_n \subseteq N_n$ which is fixed as the monodromy of the fibration varies with $n.$ 

In particular we specify a marking of the disk $\Omega = \Omega_n \subseteq S^3$. Here, by a \emph{marking} we mean an isotopy class of a diffeomorphism from the disk $\Omega_n \subseteq S^3$ to a fixed disk with $2g+1$ marked points. We let $S^3 = \mathbb{R}^3 \cup {\infty}.$ Let $\Omega \subseteq \mathbb{R}^2 \times {0}$ be the unit disk in the plane with marked points $\{1, \dots, 2g+1\}$ where the marked point $i$ is at coordinate $(\frac{-1}{2} + \frac{i-1}{2g},0,0).$
Using this marking, we identify $\Gamma_n$ and $\Delta_n$ with subsurfaces of $\Omega_n$ by flowing them along the braid until they meet $\Omega_n.$  Moreover, by fixing a simple 3-fold branched covering of the disk, we fix a marking of the fiber surface, $W = W_n,$ as shown in figure \ref{fig:cover}, by composing the marking of the disk with the branched covering.  Again, by working in the marked surface we view the subsurfaces $p^{-1}(\Gamma_n)$ and $p^{-1}(\Delta_n)$ as subsurfaces of $W = p^{-1}(\Omega).$

From our 2-dimensional perspective we observe that $\gamma_n, \delta_n, \Gamma_n,\: \text{and} \: \Delta_n$ do not depend on $n,$ so from now on we will set $\gamma = \gamma_n,$ $\delta = \delta_n,$ $\Gamma = \Gamma_n,$ and $\Delta = \Delta_n.$ Let $b_n = P\circ F^n \circ S \circ F^{-n}$ be the monodromy of $N_n,$ where $P, \: F, \: \text{and} \: S$ are the lifts of $\Pi, \: \Phi, \text{and}\: \Sigma$ respectively.

Let $c,\: d,$ and $ w_n$ denote $p^{-1}(\gamma),\: p^{-1}(\delta),\: \text{and}\: p^{-1}(\omega_n)$ respectively. Also, let $C^+$ and $D^+$ denote $p^{-1}(\Gamma)$ and $p^{-1}(\Delta)$ respectively.

\begin{figure}
    \centering
    \begin{overpic}[width=8cm]{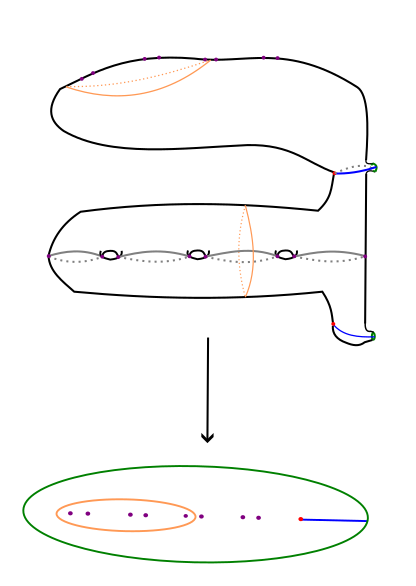}
    \put(25,70){\color{pale} \small $p^{-1}(\eta)$}
    \put(23,17){\color{pale} \small $\eta$}
    \put(43,69){\color{red} \small $p^{-1}(b)$}
    \put(45,45){\color{red} \small $p^{-1}(b)$}
    \put(48,10.5){\color{red} \small $b$}
    \put(55,13){\color{blue} \small $\alpha$}
    \put(59,18){\color{ao(english)} \small $\omega_n$}
    \put(66,72){\color{ao(english)} \small $w_n$}
    \put(65,45){\color{ao(english)} \small $w_n$}
    \put(63,67){\color{blue} \small $p^{-1}(\alpha)$}
    \put(53,38){\color{blue} \small $p^{-1}(\alpha)$}
        
    \end{overpic}
    \caption{The unique simple cover from $S_{3,2}$ to a disk. The gray curves divide the covering surface into three ``fundamental domains" for the cover. See the papers by Winarski and Fuller (\cite[section 3.2]{Winarski} and  \cite{Fuller}) for more on a similar simple branched cover.}
    \label{fig:cover}
\end{figure}

\begin{figure}
    \centering
    \begin{overpic}[width=11cm]{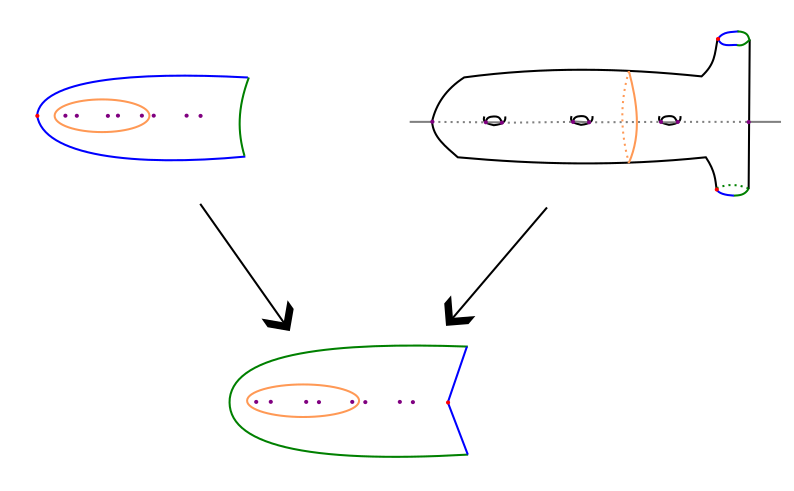}
    \put(9,55){\color{pale} \small $p^{-1}(\eta)$}
    \put(74,56){\color{pale} \small $p^{-1}(\eta)$}
    \put(38,15){\color{pale} \small $\eta$}
    \put(53,12){\color{red} \small $b$}
    \put(59,11){\color{blue} \small $\alpha$}
    \put(10,38){\color{blue} \small $p^{-1}(\alpha)$}
    \put(86,61){\color{blue} \small $p^{-1}(\alpha)$}
    \put(87,33){\color{blue} \small $p^{-1}(\alpha)$}
    \put(22,11){\color{ao(english)} \small $\omega_n$}
    \put(96,55){\color{ao(english)} \small $w_n$}
    \put(96,38){\color{ao(english)} \small $w_n$}
    
    \end{overpic}
    \caption{After cutting the disk $\Omega = \Omega_n$ along $\alpha$ and the covering surface along $p^{-1}(\alpha)$ we obtain two components of the covering surface. One component covers the cut disk trivially. The other component covers the cut disk by an order 2 involution that identifies the two boundary components.}
    \label{fig:cut_cover}
\end{figure}

\begin{lemma}\label{elev_surf}
$C^+$ and $D^+$ are each homeomorphic to the disjoint union of a disk and $S_{2,1}.$ Furthermore, $C^+ \cap D^+$ is the disjoint union of a $S_{1,2}$ and a disk.
\end{lemma}
\begin{proof}
If we restrict $p$ to a single fiber of $N_n$ we get a simple 3-fold branched cover $p': S_{g-1,2}\longrightarrow D^2$ branched over $2g+1$ points.
It follows from the work of Gabai and Kazez in \cite{GK} that 3-fold simple branched covers of the disk are unique up to equivalence. By uniqueness of the cover we may assume it takes on the configuration shown in figure \ref{fig:cover} in the disk fiber of $S^3$ immediately following $\Pi.$ That is, the distinguished branch point labeled $b$ may be taken to lie on the $(2g+1)^{\text{st}}$ strand of the braid in the fiber $\Omega_n.$ In particular, every simple closed curve that bounds a disk containing any 5 of the branch points that also does not intersect the arc $\alpha$ can be taken by a homeomorphism of the disk to the simple closed curve $\eta$ as shown in figure \ref{fig:cover}. Observe that $c$ and $d$ will consist of two simple closed curves since $\gamma$ and $\delta$ each bound a disk with an odd number of branch points. By construction, $\gamma$ and $\delta$ do not intersect the arc $\alpha,$ and hence one of the two preimages of $\gamma$ and of $\delta$ is completely contained in the upper fundamental domain of the cover and bounds a disk as in the left panel of figure \ref{fig:cut_cover}. For example, every simple closed curve that bounds a disk containing five branch points other than the distinguished branch point, $b,$ in figure \ref{fig:cover} is homeomorphic to the curve $\eta$ by the change of coordinates principle. Again by an Euler characteristic calculation we conclude that the other component is homeomorphic to $S_{2,1},$ a surface of genus two with one boundary component. Note that $\Gamma \cap \Delta$ may be taken to be a disk that intersects $\beta_n$ in 4 points. Thus, $C^+ \cap D^+$ consists of 2 disconnected components. One is the intersection of two disks in the upper fundamental domain that must be connected by the properties of the cover and hence is a disk. The other component is the double branched cover of a disk, branched over 4 points all of which have branching index 2. By another Euler characteristic calculation this is a surface of genus 1 with 2 boundary components.
\end{proof}

Observe that $b_n$ factors as $(P\circ S)\circ (S^{-1} \circ F^n \circ S)\circ F^{-n}.$ Let $C$ and $D$ denote the $S_{2,1}$ component of $C^+$ and $D^+$ respectively. To improve readability, we will frequently omit the $\circ$ symbol. We however maintain the same convention that mapping classes are multiplied according to functional notation. That is, $PS$ refers to applying first $S$ then $P.$

\begin{lemma} \label{F_ppa}
    $F$ is a partial pseudo-Anosov with $\text{supp}(F) = D$ and $\text{supp}(S^{-1}FS) = C$.
\end{lemma}
\begin{proof}
    Recall that $p^{-1}(\text{supp}(\Phi))=p^{-1}(\Gamma) = C^+.$ By lemma \ref{elev_surf} $C^+$ is the disjoint union of $C$ and a disk. Moreover, the component of $C^+$ that contains branching index one points is the disk component again by lemma \ref{elev_surf} (see also figure \ref{fig:cut_cover}). We observe that $C^+ - C$ and $D^+ - D$ are both disks and hence any mapping class restricted to either one is trivial. Note that $p$ restricted to the subsurface $C$ that contains only index two points is a double branched cover of $\Gamma.$ Since $\Phi$ is a pseudo-Anosov on $\Gamma$ it preserves two transverse singular foliations on $\Gamma.$ The foliations lift through the branched cover to give singular transverse foliations on the surface $C.$ Any possible 1-prong singularities at punctures of $\Omega_n$ are double covered in $C$ and $D$ and hence become 2-prong singularites.   Therefore, $F$ is a partial pseudo-Anosov with $\text{supp}(F) = D$ and $\text{supp}(S^{-1}FS) = C.$ See \cite[section 14.1]{primer} for more on pseudo-Anosov mapping classes arising from branched covers.
   
\end{proof}

\begin{lemma} \label{t_dist}
$b_n$ has linearly growing translation distance in the curve complexes of $C$ and $D.$ In particular there exists some $n_0>0$ such that $b_n$ is not periodic for all $n\geq n_0.$
\end{lemma}

\begin{proof}
     By lemma \ref{F_ppa}, $S^{-1}F^nS$ and $F^{-n}$ are each pseudo-Anosov on their support, namely $C$ and $D$. By lemma \ref{elev_surf} $C \cap D$ is homeomorphic to a surface with genus 1 and 2 boundary components. Hence, Theoreom 5.2 of \cite{CLM} implies that $(S^{-1}F^nS) \cdot F^{-n}$ has linearly growing translation distance in the curve complexes of each of $C$ and $D$. Therefore, there exists some $n_0$ such that for all $n\geq n_0,\: b_n$ has large translation distance in each of $C$ and $D,$ and hence is infinite order.
\end{proof}

While we have shown that our monodromies are eventually not periodic, we need monodromies that are eventually pseudo-Anosov. Since, this may not be true for $b_n$, we need to modify our manifolds $N_n$ slightly. To that end we introduce a more 2-dimensional way of discussing fibered links in closed oriented 3-manifolds. We will use the equivalent notion of an \textbf{\textit{open book decomposition}} of our closed, oriented 3-manifold $M.$ An open book decomposition of $M$ is a pair $(S,\varphi)$ where $S$ is the fiber surface ($\partial S \neq \varnothing$) and $\varphi$ is the monodromy. The mapping torus $M_{\varphi}$ of $\varphi$ is homeomorphic to a link complement in $M.$ Then $M$ is homeomorphic to the quotient of $M_{\varphi}$ under the identification $(x,t) \sim (x,t')$ where $x\in \partial S$ and $t,t'\in [0,1].$ The quotient of $\partial M_{\varphi}$ under the above identification is the fibered link in $M$ which is called the \textbf{\textit{binding}} of $(S,\varphi).$ We now describe an operation on an open book decomposition $(S,\varphi)$ of $M$ called a stabilization.

\begin{definition}
    Let $(S,\varphi)$ be an open book decomposition of $M,$ and let $(\gamma,\partial \gamma)\subseteq (S,\partial S)$ be a properly embedded arc. A \textbf{stabilization} along $\gamma$ of the open book decomposition $(S,\varphi)$ is the open book decomposition $(S',\varphi')$ of $M.$ where $S'$ is obtained by attaching a 1-handle to $S$ that connects the endpoints of $\gamma$ on $\partial S$ and $\varphi' = T_{\gamma'} \circ \varphi$ where $\gamma'$ is the curve that agrees with $\gamma$ in $S$ and intersects the co-core of the 1-handle once (see figure \ref{fig:aos}) and $T_{\gamma'}$ denotes the Dehn twist along $\gamma'.$ 
\end{definition}

\begin{figure}[H]
    \centering
    \begin{overpic}[width=8cm]{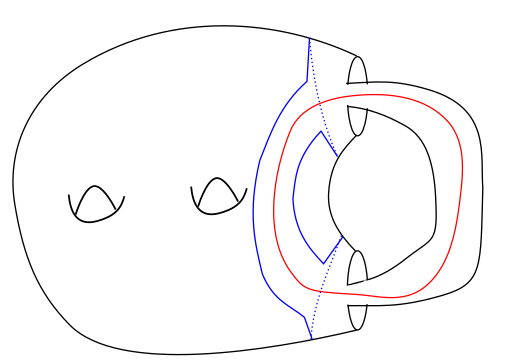}
    \put(53,55){\color{blue} $\overline{\gamma}$}
    \put(89,19){\color{red} $\gamma'$}
        
    \end{overpic}
    \caption{The surface obtained by stabilizing $S_{2,2}.$ The curve $\gamma'$ is the curve we twist along while stabilizing. The curve $\overline{\gamma}$ is the result of the "closing up" of the arc $\gamma$ which agrees with $\gamma'$ on $S_{2,2}.$}
    \label{fig:aos}
\end{figure}

The following proposition is essentially theorem 1.1 of \cite{CH}. In the same way as described in subsection \ref{sec: curves}, we associate to any arc, $\gamma,$ a collection of mutually disjoint simple closed curves, $\overline{\gamma}.$

\begin{proposition} \label{stab_pA}
    Let $(S_{g-1,2}, \varphi)$ be an open book decomposition of a closed, oriented 3-manifold $M$ such that $\varphi$ is not periodic. Let $\gamma$ be an arc connecting the two boundary components of $S_{g-1,2}.$ If $d_{S_{g-1,2}}(\overline{\gamma},\varphi(\overline{\gamma})) = N > 16$ then stabilization along the arc $\gamma$ yields an open book decomposition $(S_{g,1},T_{\gamma'} \circ \varphi)$ such that $T_{\gamma'} \circ \varphi$ is pseudo-Anosov. Recall that $\gamma'$ is the extension of $\gamma$ to the stabilized surface.
\end{proposition}

\begin{proof}
    We note here that while theorem 1.1 of \cite{CH} contains the hypothesis that the monodromy $\varphi$ is right-veering this is not needed to show that $T_{\gamma'} \circ \varphi$ is pseudo-Anosov. All that is required is that $\varphi$ is not periodic. While the lemma is essentially theorem 1.1 of \cite{CH}, the required distance between $\overline{\gamma}$ and its image under $\varphi$ is not made explicit there. For the sake of completeness, we sketch their proof and make the distance explicit. Throughout this proof, $d(\cdot, \cdot)$ will refer to $d_{S_{g-1,2}}(\cdot,\cdot).$

    Let $\varphi' = T_{\gamma'} \circ \varphi.$ We argue that $\varphi'$ is pseudo-Anosov by showing that it is neither reducible nor periodic. We will first argue that $\varphi'(\delta) \neq \delta$ for any multicurve $\delta \subseteq S_{g,1}.$ First suppose that $\delta \subseteq S_{g-1,2} \subseteq S_{g,1}.$ If $i(\varphi(\delta), \gamma')=n$ then there is a representative, $g$ of $\varphi'(\delta)$ that intersects the co-core of the 1-handle, $a,$ $n$ times. If $i(\varphi'(\delta), a)<n$ then there is a bigon consisting of a subarc of $a$ and an arc of $g.$ Then one checks that this implies there is a bigon consisting of an arc of $\gamma$ and an arc of $\varphi(\delta).$ See the third paragraph of the proof of case 1 of theorem 1.1 of \cite{CH} for this argument. This is a contradiction if $n>0.$ If $n=0$ then $i(\varphi(\beta), \overline{\gamma}) = 0$ for any component $\beta$ of $\delta.$ Hence, $d(\varphi(\beta), \overline{\gamma}) = 1$ for all such $\beta.$ However, $d(\overline{\gamma}, \varphi(\overline{\gamma})) = N$ and therefore, $d(\varphi(\overline{\gamma}), \varphi(\beta)) \geq N-1$ for every $\beta.$ Since $n=0,$ we see that $i(\varphi'(\delta),\overline{\gamma})=0.$ However, $d(\overline{\gamma}, \beta) = d(\varphi(\overline{\gamma}), \varphi(\beta)) > N-1,$ and hence $i(\overline{\gamma}, \beta) > 2^{\frac{(N-1)-2}{2}}.$ This yields contradiction so long as $N>2.$

    Now suppose that $\delta \nsubseteq S_{g-1,2}.$ Let $i(\delta,a)=k>0.$ Let $B = S_{g,1} - S_{g-1,2}$ denote the stabilization band. Normalize $\delta$ so that it intersects $\gamma'$ and $a$ transversely and efficiently. We then subdivide $\delta$ into arcs $\delta_1, \dots, \delta_k,\delta_1',\dots, \delta_k'$ so that $\delta_i \subseteq S_{g-1,2}$ and $\delta_i' \subseteq B.$ The $\delta_i'$ are all linear arcs in $B.$ If $i(\delta_i',\gamma')=0$ we call such an arc \textit{vertical}. If $i(\delta_i',\gamma')>0$ we say it has \textit{positive slope} or \textit{negative slope}. See figure \ref{fig:slopes}. We normalize and subdivide $\varphi(\delta)$ in the same way.
    Up to isotopy we may assume that there is no triangle contained in $S_{g-1,2}$ with boundary an arc of $\gamma',$ an arc of $\delta,$ and an arc of $\partial S_{g-1,2}.$ Now let $m$ denote the number of arcs of $\varphi(\delta) \cap B$ that are not vertical. Let $n$ denote the number of intersections between $\varphi(\delta)$ and $\gamma'$ that are outside of $B$ so that $i(\varphi(\delta), \gamma')=m+n.$ Note that this definition of $n$ agrees with the definition of $n$ as $i(\varphi(\delta), \gamma')$ in the above paragraph when $\varphi(\delta)$ does not intersect the band $B.$ Colin and Honda show that $i(\varphi'(\delta), a) = k \pm m+n$ where the sign depends on whether the non-vertical arcs have positive or negative slope. If $m \neq n$  then $i(\varphi'(\delta), a) \neq i(\delta,a) = k,$ so $\varphi'(\delta) \neq \delta.$ It remains to deal with the case that $m=n.$

    Suppose that $m=n.$ Then $i(\varphi'(\delta),\gamma') = m+n = 2m \leq 2k$ by the definitions of $m$ and $k.$ Suppose for the sake of contradiction that $\varphi'(\delta) = \delta.$ Then $i(\delta, \gamma') = i(\varphi'(\delta), \gamma').$ Therefore $i(\delta,\gamma') \leq 2k,$ and hence there is some $1\leq i \leq k$ such that $i(\delta_i, \gamma) \leq 2k/k = 2.$ Now we note that 
    \[i(\overline{\delta_i}, \overline{\gamma}) \leq 2+4=6. \tag{$*$}\]
    By lemma \ref{Hempel_lemma}
    \[d(\overline{\delta_i},\overline{\gamma}) \leq \lfloor 2\log_2(6)+2 \rfloor = 7.\]
    Observe that $\overline{\delta_i}$ and $\overline{\delta_j}$ cannot jointly fill the surface, so
    \[d(\overline{\delta_i},\overline{\delta_j}) \leq 2 \: \text{for all } i,j\]
    Therefore,
    \[d(\overline{\delta_j},\overline{\gamma}) \leq 2+7=9 \: \text{for all } j.\]
    Applying $\varphi$ yields
    \[d(\varphi(\overline{\delta_j}),\varphi(\overline{\gamma}))\leq 9 \: \text{for all } j.\]
    Thus,
    \[d(\varphi(\overline{\delta_j}),\overline{\gamma})\geq N - 9.\]
    Applying Hempel's lemma again,
    \[i(\varphi(\overline{\delta_j}),\overline{\gamma}) \geq 2^{(N-9)/2 - 1}.\]
   Passing from arcs to the corresponding closed curves we may increase intersection number by at most 4 (as noted in, ($*$), for example). Hence, when passing from closed curves back to arcs we may decrease intersection number by at most 4.  Therefore,
    \[i(\varphi(\delta_j),\gamma) \geq 2^{(N-9)/2 - 1} - 4.\]
    Multiplying by $k,$ we get:
    \[i(\varphi(\delta),\gamma') \geq (2^{(N-9)/2 - 1} - 4)k > 2k\]
    this is a contradiction so long as $N>16$ which proves that $\varphi'$ is not reducible.
    
    It just remains to show that $\varphi'$ is not periodic. Colin and Honda show that $\varphi'$ is not periodic by considering one of the connected components, $\eta,$ of $\partial S_{g-1,2}$ as a curve in $S_{g,1}.$ They argue that $i((\varphi')^n(\eta), a) \rightarrow \infty$ as $n \rightarrow \infty.$
\end{proof}

\begin{figure}[H]
    \centering
    \begin{overpic}[width=8cm]{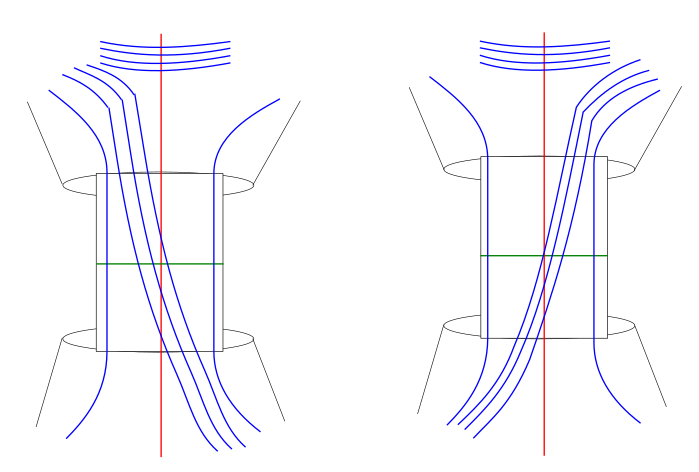}
        \put(10,60){\color{blue} \small $\delta$}
        \put(63,60){\color{blue} \small $\delta$}
        \put(18,5){\color{red} \small $\gamma'$}
        \put(80,5){\color{red} \small $\gamma'$}
        \put(34,30){\color{ao(english)} \small $a$}
        \put(89,30){\color{ao(english)} \small $a$}
    \end{overpic}
    \caption{On the left is the case that $\delta$ has negative slope, and on the right is the case that $\delta$ has positive slope. In both cases $k=5,$ $m=3,$ and $n=4.$}
    \label{fig:slopes}
\end{figure}

We now apply the above lemma to $M - w_n$ to obtain the following:
\begin{proposition} \label{Mhyp}
    Let $(S_{g-1,2},b_n)$ be the open book decomposition of $M$ corresponding to the fibered link $w_n\subseteq M,$ and let $n_0$ be the constant from lemma \ref{t_dist}. There exists an arc of stabilization $\gamma$ and $n_1\geq n_0 > 0$ such that $T_{\gamma'} \circ b_n$ is pseudo-Anosov for all $n\geq n_1,$ and hence $M - \overline{w_n}$ is hyperbolic, where $\overline{w_n}$ is the image of $w_n$ after stabilization.
\end{proposition}

\begin{proof}
      First, note that by lemma \ref{t_dist}, for all $n\geq n_0$ the mapping class $b_{n}$ is not periodic. We now explain how to choose an arc of stabilization $\gamma$ such that $d_{S_{g-1,2}}(\overline{\gamma},b_{n_0}(\overline{\gamma})) > 16$ following \cite{CH}. First let $\gamma_0$ be any properly embedded arc connecting the two boundary components of $S_{g-1,2}.$ In the case that $b_{n_0}$ is pseudo-Anosov, fix a pseudo-Anosov, $f\in \Mod(S_{g-1,2}),$ that does not share a stable or unstable lamination with $b_{n_0}.$ In the case that $b_{n_0}$ is reducible, any pseudo-Anosov, $f,$ will suffice. Let $\nu$ be the stable lamination of $f.$ Note that this implies that $b_{n_0}(\nu) \neq \nu.$ The sequence $f^i(\overline{\gamma_0}) \rightarrow \nu$ as $i\rightarrow \infty.$ We claim that $d_{S_{g-1,2}}(f^i(\overline{\gamma_0}), b_{n_0}(f^i(\overline{\gamma_0}))) \rightarrow \infty.$  If $d_{S_{g-1,2}}(f^i(\overline{\gamma_0}), b_{n_0}(f^i(\overline{\gamma_0})))$ does not approach $\infty,$ then up to passing to a subsequence $d_{S_{g-1,2}}(f^i(\overline{\gamma_0}), b_{n_0}(f^i(\overline{\gamma_0}))) = N$ for some constant $N.$ Let $f^i(\overline{\gamma_0}) = v_{i,0}, v_{i,1}, \dots , v_{i,N-1}, v_{i,N} = b_{n_0}(f^i(\overline{\gamma_0}))$ be a geodesic in the curve graph. Note that $v_{i,0}$ and $v_{i,1}$ are disjoint. Up to passing to a subsequence, $v_{i,1}$ converges to a lamination $\nu'$ such that $i(\nu,\nu')=0.$ Similarly, $v_{i,N-1}$ converges to a lamination $\lambda$ such that $i(\lambda, b_{n_0}(\nu))=0.$ Since $\nu$ is minimal and filling and $i(\nu, \nu') = 0,$ we have that $\nu = \nu'$ and similarly $b_{n_0}(\nu) = \lambda.$ Hence, we have $v_{i,1} \rightarrow \nu$ and $v_{i,N-1} \rightarrow b_{n_0}(\nu)$ as $i \rightarrow \infty,$ but $d_{S_{g-1,2}}(v_{i,1}, v_{i,N-1})<N.$ After repeating this process finitely many times, we conclude that $b_{n_0}(\nu) = \nu.$ This contradicts the choice of $f.$ Therefore, there exists some $m > 0$ such that  $d_{S_{g-1,2}}(f^m(\overline{\gamma_0}), b_{n_0}(f^m(\overline{\gamma_0}))) > 16.$ Moreover, $f$ is pseudo-Anosov and hence the action of $f$ on $\mathcal{C}(S_{g-1,2})$ has positive translation distance. In particular, the forward translates of a vertex in $\mathcal{C}(S_{g-1,2})$ are eventually arbitrarily far from any fixed vertex of $\mathcal{C}(S_{g-1,2}).$ Therefore, there exists some $l_0>0$ so that for all $l\geq l_0$ $d_{S_{g-1,2}}(f^l(\overline{\gamma_0}), \partial C) > 18.$ Let $k = \max \{ m,l \}.$ Let $\gamma$ be the arc such that $f^k(\overline{\gamma_0}) = \overline{\gamma}.$
      
      Note that although the above only shows how to choose an arc of stablization for $b_{n_0}$ we may choose such an arc of stabilization uniformly for all open book decompositions $(S_{g-1,2}, b_n)$ where $n$ is large enough. To see why this is true consider the subsurface $C.$ Observe that there exist constants $K,L>0$ such that $K^{-1}d_{C}(\overline{\gamma},S^{-1}F^{n}SF^{-n}(\overline{\gamma}))-L \leq d_{C}(\overline{\gamma},b_{n}(\overline{\gamma}))\leq Kd_{C}(\overline{\gamma},S^{-1}F^{n}SF^{-n}(\overline{\gamma}))+L$ for all $n.$ This is due to the factorization of $b_n$ as $(PS)(S^{-1}F^nSF^{-n)}$ where the $(PS)$ factor is fixed as $n \rightarrow \infty.$ By lemma \ref{t_dist} there exists $n_1 \geq n_0 >0$ such that $d_{C}(\overline{\gamma},S^{-1}F^{n}SF^{-n}(\overline{\gamma})) \gg 0$ for all $n\geq n_1.$ Hence, by theorem \ref{bgit} (see also remark \ref{bgit_rmk}) the geodesic in $\mathcal{C}(S_{g-1,2}),$ between $\overline{\gamma}$ and $b_{n}(\overline{\gamma})$ must pass within distance 1 of $\partial C$ for all $n\geq n_1.$ Therefore we have the following inequality.
    \[d_{S_{g-1,2}}(\overline{\gamma}, \partial C) - 1 \leq d_{S_{g-1,2}}(\overline{\gamma}, b_{n}({\overline{\gamma})})\]
    
    Now note that, $d_{S_{g-1,2}}(\overline{\gamma}, \partial C)$ is unchanged as $n\rightarrow \infty,$ and hence \\ $d_{S_{g-1,2}}(\overline{\gamma}, b_{n}({\overline{\gamma})})> 16$ for all $n\geq n_1.$ Thus, by proposition \ref{stab_pA} the family of open book decompositions $(S_{g,1}, T_{\gamma'} \circ b_n)$ have pseudo-Anosov monodromies for all $n\geq n_1.$    
\end{proof}

\begin{remark} \label{ppa}
    Note that stabilization does not affect the conclusions of lemmas \ref{F_ppa} and \ref{t_dist}. We consider the subsurfaces $C$ and $D$ included into the stabilized surface and the mapping class $F$ as a mapping class on the stabilized surface. Then, $T_{\gamma'} \circ b_n$ has linearly growing translation distance in the curve graphs of $C$ and $D,$ and $F$ is a partial pseudo-Anosov on $S_{g,1},$ the stabilized surface.
\end{remark}

\begin{proposition}
    In fixed genus $g,$ as $n$ tends to infinity, the knot complements $M_n = M - \overline{w_n}$ are eventually hyperbolic, with volumes tending to infinity.
\end{proposition}
\begin{proof}
     As is shown in proposition \ref{Mhyp} $b_n' = T_{\gamma'} \circ b_n$ are pseudo-Anosov for $n\geq n_1.$ Hence, their mapping tori, $M_n,$ are hyperbolic for $n\geq n_1.$ As is noted in remark \ref{ppa} and lemma \ref{t_dist}, $b_n'$ has linearly growing subsurface projections to $C_n$ and $D_n.$ Now, by applying theorem \ref{MM_dist} we see that $b_n'$ has linearly growing translation distance in the pants graph of $S_{g,1}$. By theorem \ref{brock}, the mapping tori of $b_n'$ have linearly growing volume. These mapping tori are precisely the fibered knot complements $M_n.$
\end{proof}

For fixed $g$ the family $\overline{w_n}$ are the promised family of fibered knots in $M$ satisfying the conclusion of Theorem \ref{mainthm}.

\bibliographystyle{alpha}
\bibliography{bibliography}

\end{document}